%% file: Sequeezing-function-ellipsoid-2-6-2020.tex
\documentclass[12pt, a4paper]{amsart}
\usepackage{amssymb}
\usepackage{tikz} 
\usepackage{mathabx}
\usepackage{graphicx}
\usepackage{mathrsfs}
\usepackage{hyperref}
\usepackage{here}
\usepackage{import}
\usepackage{tcolorbox}

\DeclareFontFamily{OT1}{rsfs}{}
\DeclareFontShape{OT1}{rsfs}{n}{it}{<-> rsfs10}{}
\DeclareMathAlphabet{\mathscr}{OT1}{rsfs}{n}{it}

\theoremstyle{plain}

\newtheorem{theorem}{Theorem}[section]
\newtheorem{lemma}[theorem]{Lemma}
\newtheorem{proposition}[theorem]{Proposition}
\newtheorem{corollary}[theorem]{Corollary}

\theoremstyle{definition}
\newtheorem{define}{Definition}[section]

\theoremstyle{remark}
\newtheorem{remark}{Remark}%numbering within section

\newtheorem*{acknowledgement}{Acknowledgement}

\begin{document}

%\date\today

\title[Boundary behavior of the squeezing function]{Boundary behavior of the squeezing function near a global extreme point}
\author[Ninh Van Thu, Nguyen Thi Kim Son, and Chu Van Tiep]{Ninh Van Thu\textit{$^{1,2}$}, Nguyen Thi Kim Son\textit{$^3$}, and Chu Van Tiep\textit{$^4$}}

\address{Ninh Van Thu}
\address{\textit{$^{1}$}~Department of Mathematics, Vietnam National University, Hanoi, 334 Nguyen Trai, Thanh Xuan, Hanoi, Vietnam}
\address{\textit{$^{2}$}~Thang Long Institute of Mathematics and Applied Sciences,
	Nghiem Xuan Yem, Hoang Mai, HaNoi, Vietnam}
\email{thunv@vnu.edu.vn}
\address{Nguyen Thi Kim Son}
\address{\textit{$^{3}$}~Department of Mathematics, Hanoi University of Mining and Geology, 18 Pho Vien, Bac Tu Liem, Hanoi, Vietnam}
\email{kimsonnt.0611@gmail.com}
\address{Chu Van Tiep}
\address{\textit{$^{4}$}~Department of Mathematics, The University of Danang - University of Science and Education, 459 Ton Duc Thang, Danang, Vietnam.}
\email{cvtiep@ued.udn.vn}


\begin{abstract} 
In this paper, we prove that the general ellipsoid $D_P$ is holomorphically homogeneous regular provided that it is a $WB$-domain. Then, the uniform lower bound for the squeezing function near a $(P,r)$-extreme point is also given.
\end{abstract}

\maketitle

\section{introduction}
 Let $\Omega$ be a domain in $\mathbb C^n$ and $p \in \Omega$. For a holomorphic embedding $f\colon \Omega \to \mathbb B^n:=B(0;1)$ with $f(p)=0$, we set
$$\sigma_{\Omega,f}(p):=\sup\left \{r>0\colon B(0;r)\subset f(\Omega)\right\},$$
where $ B(z;r)\subset\mathbb{C}^n$  denotes the ball of radius $r$ with center at $z$. Then the \textit{squeezing function} $\sigma_{\Omega}\colon \Omega\to\mathbb R$ is defined in \cite{DGZ12} as
$$
\sigma_{\Omega}(p):=\sup_{f} \left\{\sigma_{\Omega,f}(p)\right\}.
$$
Notice that $0 < \sigma_{\Omega}(z)\leq 1$ for any $z \in \Omega$ and the squeezing function is obviously invariant under biholomorphisms. A domain $\Omega$ in $\mathbb C^n$ is called \emph{holomorphically homogeneous regular} (HHR) if $\inf_{z\in \Omega}\sigma_\Omega(z)>0$ (cf. \cite{LSY04,Yeung2009}, for example). It is worthwhile to note that the HHR property is clearly preserved by biholomorphisms and implies the completeness and the equivalence of the classical holomorphic invariant metrics such as the Carath\'{e}odory metric, the Kobayashi metric, the Bergman metric, and the K\"{a}hler-Einstein metric (see \cite{DGZ16, LSY04, Yeung2009}). Moreover, it is known that the class of  HHR domains includes bounded homogeneous domains, bounded convex domains, and $\mathcal{C}^2$-smooth bounded strongly pseudoconvex domains (see e.g., \cite{DWZZ20} and  the references therein). 

As a main result of this paper, we first present a new class of HHR domains, for instance, a nonconvex domain $\{(z_1,z_2)\in \mathbb C^2\colon |z_2|^2+|z_1|^8+|z_1|^2\mathrm{Re}(z_1^6)/2<1\}$ can be its typical element. To do this, for positive integers $m_1,\ldots, m_{n-1}$ let us consider a general ellipsoid $D_P$ in $\mathbb C^n\;(n\geq2)$, defined  by
\begin{equation*}
\begin{split}
D_P :=\{(z',z_n)\in \mathbb C^n\colon |z_n|^2+P(z')<1\},
\end{split}
\end{equation*}
where $P(z')$ is a $(1/m_1,\ldots,1/m_{n-1})$-homogeneous polynomial given by 
\begin{equation*}\label{series expression}
P(z')=\sum_{wt(K)=wt(L)=1/2} a_{KL} {z'}^K  \bar{z'}^L,
\end{equation*}
where $a_{KL}\in \mathbb C$ with $a_{KL}=\bar{a}_{LK}$, satisfying that $P(z')>0$ whenever $z'\ne 0$. Here and in what follows, $z':=(z_1,\ldots,z_{n-1})$ and $ wt(K):=\sum_{j=1}^{n-1} \frac{k_j}{2m_j}$ denotes the weight of any multi-index $K=(k_1,\ldots,k_{n-1})\in \mathbb N^{n-1}$ with respect to $\Lambda:=(1/m_1,\ldots,1/m_{n-1})$. (See Section \ref{section3} for the definition of weighted homogeneous polynomials.) The domain $D_P$ is called a WB-domain if $D_P$ is strongly pseudoconvex at every boundary point outside the set $\{(0',e^{i\theta})\colon \theta\in \mathbb R\}$ (cf. \cite{AGK16}).

Our first main result is the following theorem.
\begin{theorem}\label{Sqtheorem} Let $P$ be a weighted homogeneous polynomial with weight $(m_1,\ldots,m_{n-1})$ given by  \eqref{series expression} such that $P(z')> 0$ for all $z'\in \mathbb C^{n-1}\setminus\{0'\}$. If $D_P$ is a WB-domain, then $D_P$ is HHR.
\end{theorem}

Thanks to the explicit description for the automorphism group of $D_P$, denoted by $\mathrm{Aut}(D_P)$, given in \cite{NNTK19}, we shall deduce the lower estimation of the squeezing function of $D_P$ to the subset $\{(z',0)\in D_P\colon P(z')=1\}$, whose closure intersects $\partial D_P$ at only strongly pseudoconvex boundary points. Then our proof of Theorem \ref{Sqtheorem} becomes a simple consequence of Theorem \ref{tendto1} as shown in Section \ref{section2}. 

As a consequence of Theorem \ref{Sqtheorem}, by \cite{DGZ16, LSY04, Yeung2009}, one obtains the following corollary.
\begin{corollary} The Carath\'{e}odory metric, the Kobayashi metric, the Bergman metric and the K\"{a}hler-Einstein metric on $D_P$ are all equivalent and complete, provided that $D_P$ is a WB-domain.
\end{corollary}

Next, as our second main result, we shall discuss the behavior of the squeezing function near a global extreme point. 
Let $\Omega$ be a domain in $\mathbb C^n$ with $\mathcal{C}^2$-smooth boundary near a boundary point $p\in \partial \Omega$. Let us recall that a boundary point $p$ is said to be \emph{spherically extreme} if there exists a ball $B(c(p);R)$ in $\mathbb C^n$ of some radius $R$, centered at some point $c(p)$ such that $\Omega\subset B(c(p); R)$ and $p\in \partial \Omega\cap \partial B(c(p); R)$ (see \cite{KZ16}). It was shown that $ \lim_{\Omega\ni q\to p}\sigma_\Omega(q)=1$ if $p$ is a spherically extreme boundary point of $\Omega$ (see \cite[Theorem $3.1$]{KZ16}). However, a spherically extreme boundary point does not exist generally, even for a smooth pseudoconvex domain of finite type. For example, the boundary point $(0',1)\in \partial D_P$ is not spherically extreme if $\min\{m_1,\ldots,m_{n-1}\}\geq 2$. 

Recently, for a smooth pseudoconvex domain of finite type, K. Diederich et al. \cite{DFW14} proved that for any bounded domain $\Omega\subset \mathbb C^n$, which is locally convexifiable and has finite $1$-type $2k$ at $p\in \partial \Omega$, having a Stein neighborhood basis, there exists a holomorphic embedding $f\colon \overline{\Omega}\to \mathbb C^n$ such that $f(p)$ is a global extreme point of type $2k$ for $\overline{f(\Omega)}$ in the sense that $f(\Omega)\subset \mathbb B_k^n:=\{z\in \mathbb C^n\colon |z_n|^2+\|z'\|^{2k}<1\}$ and  $f(p)=(0',1)$. Therefore, this motivates us to introduce the following definition.
\begin{define} Let $\Omega$ be a domain in $\mathbb C^n$ ($n\geq 2$), $p\in \partial \Omega$, $P(z')$ be the polynomial given in (\ref{series expression}), and $r\in (0,1]$. We say that $p$ is \emph{$(P,r)$-extreme} if there exists a holomorphic embedding $f\colon \overline{\Omega}\to \overline{E_P}$ such that $f(p)=(0',0)$ and $D(r)\subset f(\Omega)$, where 
\begin{equation*}
\begin{split}
E_P&:=\{(z',z_n)\in \mathbb C^n\colon P(z')<2\ \mathrm{Re}(z_n)\};\\
D(r)&:=D_{P,r}=\{(z',z_n)\in \mathbb C^n\colon |z_n-r|^2+P(z')<r^2\}.
\end{split}
\end{equation*}
In this situation, denote by $\Gamma(r,c):=f^{-1}(D(r)\cap \{(z',z_n)\in \mathbb C^n\colon|\mathrm{Im}(z_n)|\leq c |\mathrm{Re}(z_n)|\}$) for $c>0$.
\end{define} 
\begin{remark} Since $D(r')\subset D(r)$ for $0<r'<r\leq 1$, it follows that if $p$ is $(P,r)$-extreme, then it is also $(P,r')$-extreme for any $0<r'<r\leq 1$. In addition, $p$ is spherically extreme if and only if it is $(\|z'\|^2,r)$-extreme for some $0<r\leq 1$.
\end{remark}

Our second main result is the following theorem.
\begin{theorem}\label{maintheorem} Let $\Omega$ be a domain in $\mathbb C^n$ $(n\geq 2)$ and $p\in \partial \Omega$ be a $(P,r)$-extreme point with $0<r\leq 1$. Then, for any $0<r'<r$ and $c>0$ there exist $\epsilon_0, \gamma_0>0$ such that
$$
\sigma_\Omega(q)>\gamma_0,\; \forall q\in \Gamma(r',c)\cap B(p;\epsilon_0).
$$
\end{theorem}
\begin{remark}\label{nontangent} The convergence of a sequence of points in $\Gamma(r,c)$ to $p$ is exactly the $\Lambda$-nontangential convergence introduced in \cite[Definition $3.4$]{NN19}. For the case when $\{a_j\}\subset f^{-1}(D(r))\subset \Omega$ does not converge $\Lambda$-nontangentially to $p=0$, i.e., for any $0<r'<r$ and $c>0$ there exists $j_{r',c}\in \mathbb N$ such that $a_{j}\not \in \Gamma(r',c)$ for all $j\geq j_{r',c}$, we do not know the behavior of $\{\sigma_\Omega(a_j)\}$. Therefore, a natural question to ask is whether $\liminf\limits_{f^{-1}(D(r))\ni z\to p}\sigma_\Omega(z)>0$ holds.
\end{remark}

The organization of this paper is as follows: In Section~\ref{section2} we provide several properties of the squeezing function needed for a proof of Theorem \ref{Sqtheorem}. In Section \ref{section3}, we recall the explicit description for $\mathrm{Aut}(D_P)$ given in \cite{NNTK19} and the proof of Theorem \ref{Sqtheorem} is given. Then, we shall prove Theorem \ref{maintheorem} in detail in Section \ref{section4}. In addition, an alternative proof of Theorem \ref{Sqtheorem} will be presented in Appendix.
\section{Several properties of the squeezing function}\label{section2}
\smallskip
Let $\Omega$ be a bounded domain in $\mathbb C^n$. Denote $r(z,\Omega)$ and $R(z,\Omega)$, respectively, by
$$
r(z,\Omega)=\sup\{r>0\colon B(z;r)\subset \Omega\}\;\text{and}\; R(z,\Omega)=\inf\{R>0\colon B(z;R)\supset \Omega\}. 
$$
For a subset $K\subset \Omega$, we put 
$$
r(K,\Omega)=\inf_{z\in K} r(z,\Omega);\; R(K,\Omega)=\sup_{z\in K} R(z,\Omega).
$$
Now we prepare a technical lemma.
\begin{lemma}\label{interiorpoint-sq} Let $\Omega$ be a bounded domain in $\mathbb C^n$ and $K$ be a relative compact subset of $ \Omega$. Then one has $r(z,\Omega)/R(z,\Omega)\leq \sigma_\Omega(z)\leq 1$ for any $z\in \Omega$ and $\inf\limits_{z\in K} \sigma_\Omega(z)\geq \dfrac{r(K,\Omega)}{R(K,\Omega)}>0$.
\end{lemma}
\begin{proof} Let us consider an affine map $f(\zeta):=\dfrac{\zeta-z}{R(z,\Omega)}$. Then one sees that $f$ is a biholomorphic map from $\Omega$ into $\mathbb B^n$. Moreover,  we have
$$
B\left(0; r(z,\Omega)/R(z,\Omega)\right)\subset  f(B(z;r(z,\Omega)))\subset f(\Omega),
$$
and hence 
$$
\inf_{z\in K} \sigma_\Omega(z)\geq \inf_{z\in K} \dfrac{r(z,\Omega)}{R(z,\Omega)}=\dfrac{r(K,\Omega)}{R(K,\Omega)}>0.
$$
Therefore, the proof is complete.
\end{proof}
\begin{define} Let $\Omega$ be a bounded domain in $\mathbb C^n$ and $\Sigma$ be a subset of $\Omega$. We say that $\Omega$ is \emph{HHR on $\Sigma$} if $\inf_{z\in \Sigma}\sigma_\Omega(z)>0$. In particular, $\Omega$ is \emph{HHR} if it is HHR on $\Omega$.
\end{define} 
Next, we prepare a proposition which is crucial in our proof of Theorem \ref{Sqtheorem}.
\begin{proposition}\label{newtheorem} Let $\Omega$ be a bounded domain  in $\mathbb C^n$. Suppose that there exists a subset $\Sigma\subset \Omega$ satisfying that $\forall z\in \Omega \; \exists f\in \mathrm{Aut}(\Omega)$ such that $f(z)\in \Sigma$. Then $\Omega$ is HHR if it  is HHR on $\Sigma$.
\end{proposition}
\begin{proof}
Suppose that $\Omega$ is HHR on $\Sigma$. By definition, there exists $c>0$ such that $\sigma_\Omega(z)\geq c$ for all $z\in \Sigma$. Now let $ z\in \Omega$ be arbitrary. We shall prove that $\sigma_\Omega(z)\geq c$. Indeed, by assumption there exists an automorphism $f\in \mathrm{Aut}(\Omega)$ such that $f(z)\in \Sigma$ and thus $\sigma_\Omega(z)=\sigma_\Omega(f(z))\geq c$ because of the invariance of the squeezing function under biholomorphisms. Hence, the proof is complete.
\end{proof}
Let us recall that a compact set $K\Subset \mathbb C^n$ is said to have a Stein neighborhood basis if for any domain $V$ containing $K$ there exists a pseudoconvex domain $\Omega_V$ such that $K\subset \Omega_V\subset V$. For instance, the closure $\overline{\Omega}$ of a smooth bounded pseudoconvex domain $\Omega$ in $\mathbb C^n$ has a (strong) Stein neighborhood basis if $\Omega$ has a defining function $\rho$ and there exists $\epsilon_0>0$ such that $\{z\in\mathbb C^n \colon \rho(z)<\epsilon\}$ is pseudoconvex for all $\epsilon\in [0,\epsilon_0]$ (cf. [Sa12]).

Recently, the behavior of the squeezing function near a strongly pseudoconvex boundary point was established.
\begin{theorem}[\cite{ DGZ16, DFW14,KZ16}]\label{tendto1}
 If a bounded domain $\Omega$ in $\mathbb C^n$ admitting a Stein neighborhood basis has a strongly pseudoconvex boundary point, say $p$, then $\lim_{z\to p}\sigma_\Omega(z)=1$. 
\end{theorem}
As a consequence, Theorem \ref{tendto1}, together with Lemma \ref{interiorpoint-sq} and Proposition \ref{newtheorem}, implies the following:
\begin{corollary}\label{sqcorollary}Let $\Omega$ be a bounded domain in $\mathbb C^n$ admitting a Stein neighborhood basis. Suppose that there exists a subset $M\subset \Omega$ satisfying that $\forall z\in \Omega \; \exists f\in \mathrm{Aut}(\Omega)$ such that $f(z)\in M$.   If each $p\in \overline{M}\cap \partial \Omega$ is strongly pseudoconvex, then $\Omega$ is HHR. 
\end{corollary}
\begin{proof}
By Proposition \ref{newtheorem}, it suffices to show that $\Omega$ is HHR on $M$. Indeed, suppose otherwise that there exists a sequence $\{z_j\}\subset M$ such that $\sigma_\Omega(z_j)\to 0$ as $j\to \infty$. Taking a subsequence if necessary we may assume that either $z_j\to p\in \overline{M}\cap \partial \Omega$ or $\{z_j\}\Subset \Omega$. By virtue of Lemma \ref{interiorpoint-sq}, the latter case does not occur. On the other hand, for the former case one has $\sigma_\Omega(z_j)\to1$ as $j\to \infty$ by Theorem \ref{tendto1}, which is a contradiction. This ends the proof.

\end{proof}

\begin{remark} By Lemma \ref{interiorpoint-sq}, it is easy to see that $\Omega$ is HHR on any relative compact subset  $K\Subset \Omega$.  In addition, one can infer that if $\sigma_\Omega(p)=1$ for some  $p\in\Omega$, then $\Omega$ is biholomorphic to the unit open ball (cf.  \cite{DGZ12}). Moreover, it follows from Corollary \ref{sqcorollary} that $\Omega$ is HHR if each $p\in \overline{\Omega/\mathrm{Aut}(\Omega)}\cap \partial \Omega$ is strongly pseudoconvex.
\end{remark}

\section{Squeezing function of  the general ellipsoid} \label{section3}
%\subsection{Automorphism group of a general ellipsoid}
\smallskip
Let us assign weights $\frac{1}{2m_1}, \ldots,\frac{1}{2m_{n-1}}, 1$ to the variables $z_1,\ldots, z_{n-1}, z_n$, respectively and denote by $wt(K):=\sum_{j=1}^{n-1} \frac{k_j}{2m_j}$ the weight of an $(n-1)$-tuple $K=(k_1, \ldots, k_{n-1})\in \mathbb Z^{n-1}_{\geq0}$. A real-valued polynomial $P$ on $\mathbb C^{n-1}$ is called a \emph{weighted homogeneous polynomial with weight} $(m_1, \ldots ,m_{n-1})$~(or simply \emph{$(1/m_1,\ldots,1/m_{n-1})$-homogeneous}), if
$$
P(t^{1/2m_1}z_1, \ldots,t^{1/2m_{n-1}}z_{n-1}) = t P(z_1, \ldots , z_{n-1})~\text{for all}~ z'\in \mathbb C^{n-1} ~\text{and}~t>0.
$$
In the case when $m = m_1 =\cdots = m_{n-1}$, then $P$ is called \emph{homogeneous of degree} $2m$. We note that if $P(z')$ is a $(1/m_1,\ldots,1/m_{n-1})$-homogeneous polynomial, then 
\begin{equation*}\label{2018-0}
P(z')=\sum_{wt(K)+wt(L)=1} a_{KL} {z'}^K  \bar{z'}^L,
\end{equation*}
where $a_{KL}\in \mathbb C$ with $a_{KL}=\bar{a}_{LK}$~(see \cite{NNTK19}). 

Throughout this paper,  let $P(z')$ be a $(1/m_1,\ldots,1/m_{n-1})$-homogeneous polynomial given by 
\begin{equation}\label{series expression}
P(z')=\sum_{wt(K)=wt(L)=1/2} a_{KL} {z'}^K  \bar{z'}^L,
\end{equation}
where $a_{KL}\in \mathbb C$ with $a_{KL}=\bar{a}_{LK}$, satisfying that $P(z')>0$ whenever $z'\ne 0'$. In addition, since $P(z')>0$ for $z'\ne 0'$ and by the weighted homogeneity, there are two constants $c_1,c_2>0$ such that
$$
c_1 \sigma_\Lambda(z') \leq P(z')\leq c_2 \sigma_\Lambda(z'),
$$
where $\sigma_\Lambda(z')=\sum_{j=1}^{n-1}|z_j|^{2m_j}$ (cf. \cite[Lemma 6]{NNTK19}, \cite{Yu95}). Furthermore, one sees that $\overline{D_P}$ has a Stein neighborhood basis.

We first consider the general ellipsoid $D_P$ and the model $E_P$ in $\mathbb C^n\;(n\geq2)$, defined respectively by
\begin{equation*}
\begin{split}
D_P &:=\{(z',z_n)\in \mathbb C^n\colon |z_n|^2+P(z')<1\};\\
E_P&:=\{(z',z_n)\in \mathbb C^n\colon P(z')<2\ \mathrm{Re}(z_n)\}.
\end{split}
\end{equation*}
Then, we need the following lemma which is essentially well-known~(cf.~\cite{BP94}).
\begin{lemma}\label{automorphism-ellipsoid}
Let $P$ be a weighted homogeneous polynomial with weight $(m_1,\ldots,m_{n-1})$ given by  \eqref{series expression} such that $P(z')>0$ for all $z'\in \mathbb C^{n-1}\setminus\{0'\}$. Then,  the holomorphic map $\psi $ defined by
\[
(z',z_n)\mapsto  \left( \frac{2^{1/2m_1}}{(1+z_n)^{1/m_1}} z_1,\ldots,  \frac{2^{1/2m_{n-1}}}{(1+z_n)^{1/m_{n-1}}} z_{n-1},  \frac{1-z_n}{1+z_n}\right),
\]
is a biholomorphism from $D_P$ onto $E_P$.
\end{lemma}
\begin{proof} Indeed, a direct computation shows that
\begin{equation*}
\begin{split}
\mathrm{Re}\left(\frac{1-z_n}{1+z_n} \right)&=\mathrm{Re}\left(\frac{(1-z_n)(1+\bar z_n)}{|1+z_n|^2} \right)\\
&=\frac{1-|z_n|^2}{|1+z_n|^2}.
\end{split}
\end{equation*}
Moreover, since $P$ has the form as in \eqref{series expression}, it follows that
\[
P \left( \frac{2^{1/2m_1}}{(1+z_n)^{1/m_1}} z_1,\ldots,  \frac{2^{1/2m_{n-1}}}{(1+z_n)^{1/m_{n-1}}} z_{n-1}\right)=\frac{2}{|1+ z_n|^2} P(z').
\]
 Therefore, one can deduce that
\[
 2\mathrm{Re}\left(\frac{1-z_n}{1+z_n} \right) - P \left( \frac{2^{1/2m_1}}{(1+z_n)^{1/m_1}} z_1,\ldots,  \frac{2^{1/2m_{n-1}}}{(1+z_n)^{1/m_{n-1}}} z_{n-1}\right)   >0 
\]if and only if 
\[
|z_n|^2+P(z')<1.
\]
Hence, the conclusion easily follows from the previous relation.
\end{proof}
\begin{remark} A direct computation also shows that $\psi^{-1}=\psi$, $\psi(0',0)=(0',1)$ and $\psi(0',1)=(0',0)$. In addition, $\psi(z)\to  (0',-1)$ as $ E_P\ni z\to \infty$.
\end{remark} 
Next, the first author et al. \cite{NNTK19} pointed out the following lemma.
\begin{lemma}[see Lemma 7 in \cite{NNTK19}]\label{dilation-property} Let $P$ be a weighted homogeneous polynomial with weight $(m_1,\ldots,m_{n-1})$ given by  \eqref{series expression} such that $P(z')> 0$ for all $z'\in \mathbb C^{n-1}\setminus\{0'\}$. Then, $\mathrm{Aut}(D_P)$ contains the following automorphisms $\phi_{a,\theta}$, defined by
\begin{equation}\label{formautomorphism}
(z',z_n)\mapsto  \left( \frac{(1-|a|^2)^{1/2m_1}}{(1-\bar{a}z_n)^{1/m_1}} z_1,\ldots,  \frac{(1-|a|^2)^{1/2m_{n-1}}}{(1-\bar{a}z_n)^{1/m_{n-1}}} z_{n-1}, e^{i\theta} \frac{z_n-a}{1-\bar{a}z_n}\right),
\end{equation}
where $a\in\Delta:=\{z\in \mathbb C\colon |z|<1\}$ and $\theta\in \mathbb R$. In addition, $\mathrm{Aut}(E_P)$ contains the dilations $\Lambda_\lambda$, $\lambda>0$, defined by
\begin{align*}
\Lambda_{\lambda}(z',z_n)=\left(\frac{z_1}{\lambda^{1/2m_1}},\ldots,\frac{z_{n-1}}{\lambda^{1/2m_{n-1}}},\frac{z_n}{\lambda}\right).
\end{align*}
\end{lemma}

%\subsection{Squeezing function of  the general ellipsoid}
Now we are ready to prove Theorem \ref{Sqtheorem}.
\begin{proof}[Proof of Theorem \ref{Sqtheorem}]
Denote by $\Sigma:=\{(z',0)\in D_P\colon P(z')=1\}$. Then for each $p=(p',p_n)\in D_P$, one has $\phi_{p_n, 0}(p)\in \Sigma$, where $\phi_{p_n, 0}$ is an automorphism given in (\ref{formautomorphism}). Since each boundary point of $\overline{\Sigma}\cap \partial D_P$ is strongly pseudoconvex and the asymptotic limit value of the squeezing function is $1$ by Theorem \ref{tendto1}, it follows that $\sigma_{D_P}(z)$ is uniformly bounded away from zero on $\Sigma$. Therefore, the assertion finally follows by Proposition \ref{newtheorem}.
\end{proof}
\section{Proof of Theorem \ref{maintheorem}}\label{section4}
\smallskip
This section is fully devoted to the proof of Theorem \ref{maintheorem}. Let $\Omega$ be a domain described in the hypothesis of Theorem \ref{maintheorem}. That is, $p\in \partial \Omega$ is a $(P,r)$-extreme point with $0<r\leq 1$. Then, by virtue of the invariance of the squeezing function under biholomorphisms, we may assume that $D(r)\subset \Omega\subset E_P$ and $p=(0',0)$. Let us fix $0<r'<r$, $c>0$. Then $\Gamma(r',c)$ becomes $\Gamma(r',c)=D(r')\cap \{(z',z_n)\in \mathbb C^n\colon|\mathrm{Im}(z_n)|\leq c |\mathrm{Re}(z_n)|\}$.
\begin{center}
\begin{figure}[ht]

	\def\svgwidth{0.6\columnwidth}
	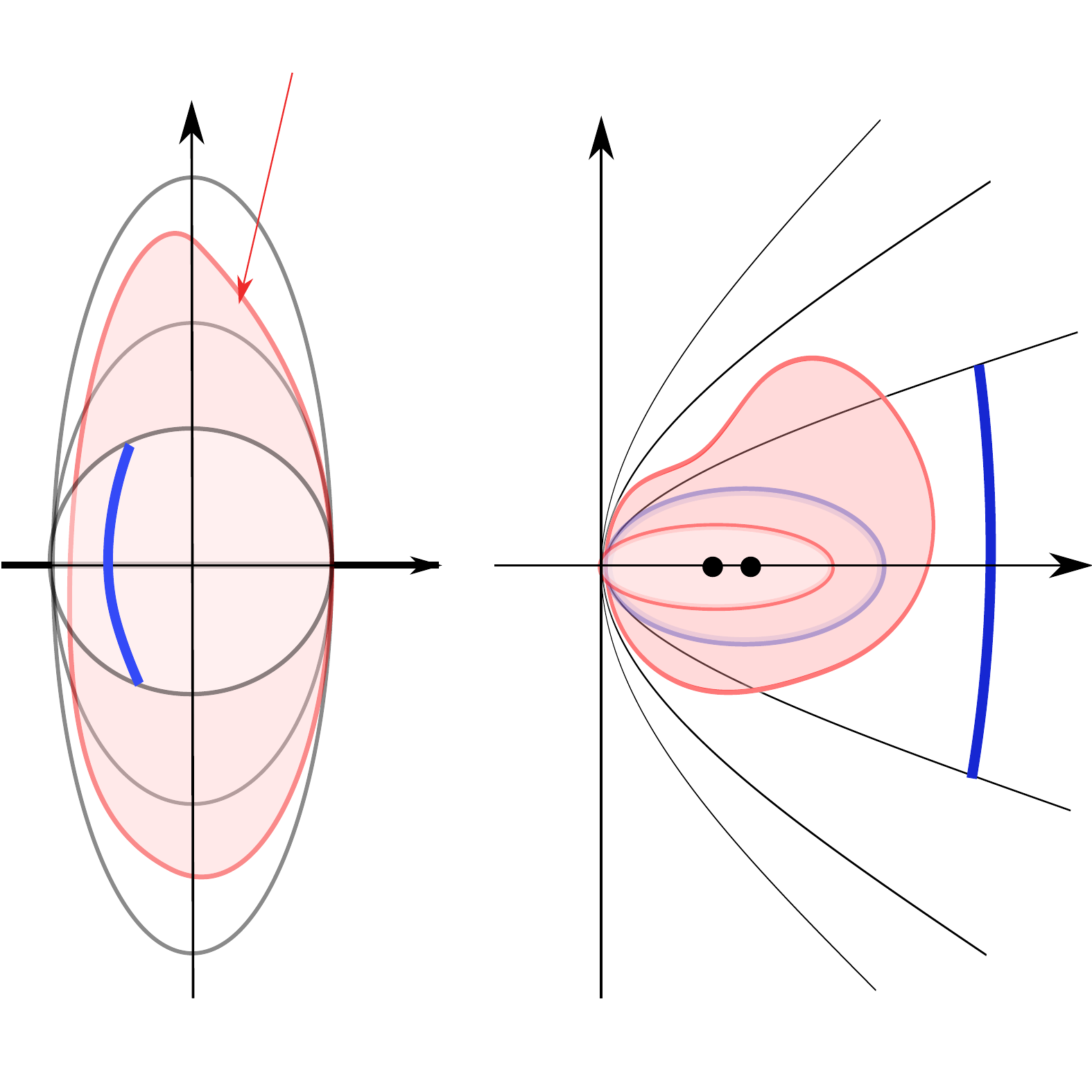
%//{{\bf Figure 1}}
\caption{}\label{Figure1}
\end{figure}
\end{center}
 Let us denote by $E^r:=E_{P/r}, E:=E^1=E_P$, $D^r:=\{(z',z_n)\in\mathbb C^n\colon |z_n|^2+P(z')/r<1\}, D:=D^1=D_P$ for simplicity.  Then one can see that $ D(r')\subset D(r)\subset \Omega\subset E$ and $D(r')\subset E^{r'}$. Moreover, the biholomorphism $\psi$, given in Lemma \ref{automorphism-ellipsoid},  maps $E, E^r, E^{r'}$ onto $D, D^r, D^{r'}$, respectively (see Lemma \ref{automorphism-ellipsoid} and Figure \ref{Figure1} above). Furthermore, by Lemma \ref{dilation-property}, $E, E^r, E^{r'}$ are invariant under the action of $\Lambda_\lambda$, defined by
 $$
\Lambda_{\lambda}(z):=\left(\frac{z_1}{\lambda^{1/2m_1}},\ldots,\frac{z_{n-1}}{\lambda^{1/2m_{n-1}}},\frac{z_n}{\lambda}\right),\quad \lambda>0.
$$
In addition, a straightforward calculation shows that
$$
\Lambda_{\lambda}(D(r))= \left\{(z',z_n)\in \mathbb C^n\colon \lambda |z_n|^2+P(z')<2r\mathrm{Re}(z_n) \right\},
$$
which converges to $E^r$ as $\lambda\to 0^+$. Moreover, $\psi\circ \Lambda_{\lambda}(D(r))$ converges to $D^r$ as $\lambda\to 0^+$.
  Here and in what follows, a family of domains $\{X_\lambda\}_{\lambda\in \mathbb R^+}$ is said to \emph{converge to a domain $X$ as $\lambda\to 0^+$} if for each compact subset $K\Subset X$, there exists $\lambda_0=\lambda_0(K)$ such that $K\subset X_\lambda$ for all $0<\lambda< \lambda_0$.

Now let us define a set 
\begin{align*}
\widetilde\Sigma&:=\left\{z=(z',z_n)\in E^{r'}\colon  \|z\|=1, |\mathrm{Im}(z_n)|\leq c |\mathrm{Re}(z_n)|\right\}\\
&=\left\{z=(z',z_n)\in \mathbb C^n\colon  P(z')<2r'\mathrm{Re}(z_n),\|z\|=1,  |\mathrm{Im}(z_n)|\leq c |\mathrm{Re}(z_n)|\right\}\Subset E^r
\end{align*}
and then $ \Sigma:=\psi(\widetilde \Sigma)\subset D^{r'}$. Then, one has $\widetilde \Sigma\Subset E^{r}$, and thus $\Sigma\Subset D^{r}$. In particular, for any $q\in E^{r'}$ with $|\mathrm{Im}(q_n)|<c |\mathrm{Re}(q_n)|$, the orbit $\{\Lambda_\lambda(q)\}_{\lambda\in \mathbb R^+}$ meets the set $\widetilde\Sigma$ in a (unique) point. 

Let $q\in \Gamma(r',c) \cap B(0;\epsilon_0)\subset E^{r'}\cap B(0;\epsilon_0)$ be arbitrary, where $\epsilon_0>0$ will be chosen later. Then, there exists $\lambda>0$ such that $\Lambda_\lambda (q)\in \widetilde \Sigma$, i.e., $\|\Lambda_\lambda (q) \|=1$. Notice that $\lambda\to 0^+$ whenever $q\to p=(0',0)$. In addition, since $\psi\circ \Lambda_{\lambda}(D(r))$ converges to $D^r$ as $\lambda\to 0^+$, it follows that $\epsilon_0>0$ can be chosen so that we always have $\Sigma\Subset \psi\circ\Lambda_{\lambda}(D(r))\subset \psi\circ\Lambda_{\lambda}(\Omega)$ and
$$
\mathrm{dist}\left( \Sigma, \partial \big(\psi\circ\Lambda_{\lambda}(\Omega)\big)\right)\geq \mathrm{dist}\left( \Sigma, \partial \big(\psi\circ\Lambda_{\lambda}(D(r)\big)\right)>\mathrm{dist}(\Sigma, \partial D^r)/2>0
$$
for any $q\in D(r')\cap B(0;\epsilon_0)$. Here and in what follows, $\mathrm{dist}(A,B)$ denotes the Euclidean distance between two subsets $A,B \subset \mathbb C^n$.

In summary, the biholomorphism $G_\lambda:=\psi\circ \Lambda_\lambda$ from $E$ onto $D$ satisfies the following properties:
\begin{itemize} 
\item[a)] $G_\lambda(\Omega)\subset D$;
\item[b)] $G_\lambda(q)\in \Sigma \Subset D^r$;
\item[c)] $\mathrm{dist}\left( G_\lambda(q) ,\partial \big(G_\lambda(\Omega)\big)\right)>\delta$,
\end{itemize}
where $\delta:=\mathrm{dist}(\Sigma, \partial D^r)/2$. Therefore, by Lemma \ref{interiorpoint-sq} and again by the invariance of the squeezing function under biholomorphisms, we conclude that 
$$
\sigma_\Omega(q)=\sigma_{G_\lambda(\Omega)}(G_\lambda(q))>\delta/d>0, \;\forall\; q\in \Gamma(r',c) \cap B(0;\epsilon_0),
$$
where $d$ denotes the diameter of $D_P$. This ends the proof of Theorem \ref{maintheorem} with $\gamma_0:=\delta/d$. \hfill $\Box\;$
\section*{Appendix}
In this Appendix, we shall present an alternative proof of Theorem \ref{Sqtheorem} by the following argument. Indeed, suppose otherwise that there exists a sequence $\{a_j=(a_j',a_{jn})\}\subset D_P$ such that $a_j\to p\in \partial D_P$ and $\sigma_{D_P}(a_j)\to 0$ as $j\to \infty$. If $p$ is strongly pseudoconvex, then $\sigma_{D_P}(a_j)\to 1$ as $j\to \infty$ by Theorem \ref{tendto1}. Hence our conclusion immediately follows from this case. Moreover, since $D_P$ is a WB-domain, it suffices to assume that $p=(0',e^{i\theta})$, $\theta\in \mathbb R$, which is weakly pseudoconvex. In this case, one must have $a_j'\to 0'$ and $a_{jn}\to e^{i\theta}$ as $j\to \infty$. Denote by $\rho(z):=|z_n|^2-1+P(z')$ a local defining function for $D_P$. Then $\mathrm{dist}(a_j, \partial D_P)\approx -\rho(a_j)\approx 1-|a_{jn}|^2-P(a_j')$. Here and in what follows, we use symbols $\lesssim$ and $\gtrsim$ to denote inequalities up to a positive multiplicative constant. In addition, we use a symbol $\approx $ for the combination of $\lesssim$ and $\gtrsim$. 

Let us denote by $\psi_j:=\phi_{a_{jn},0}$ the automorphism of $D_P$ given in Lemma \ref{dilation-property}. Then $\psi_j\in \mathrm{Aut}(D_P)$, given by
$$
\psi_j(z)=\left( \frac{(1-|a_{jn}|^2)^{1/2m_1}}{(1-\bar{a}_{jn}z_n)^{1/m_1}} z_1,\ldots,  \frac{(1-|a_{jn}|^2)^{1/2m_{n-1}}}{(1-\bar{a}_{jn}z_n)^{1/m_{n-1}}} z_{n-1}, \frac{z_n-a_{jn}}{1-\bar{a}_{jn} z_n}\right),
$$
and hence $\psi_j(a_j)=(b_j,0)$, where 
$$
b_j=  \left( \frac{a_{j1}}{(1-|a_{jn}|^2)^{1/2m_1}} ,\ldots, \frac{a_{j (n-1)}}{(1-|a_{jn}|^2)^{1/2m_{n-1}}}\right). 
$$
Thanks to the boundedness of  $\{b_j\}$, without loss of generality we may assume that $b_j\to b\in \mathbb C^{n-1}$ as $j\to \infty$.

We now divide the argument into two cases as follows:
\smallskip

\noindent
{\bf Case 1.} The sequence $\{a_j\}$ converges  $\Lambda$-nontangentially to $p$ (cf. Remark \ref{nontangent}).  Then $P(a_j')\lesssim \mathrm{dist}(a_j, \partial D_P)$. Without loss of generality, we may assume that $P(a_j')\leq C(1-|a_{jn}|^2-P(a_j'))$ for some $C>0$. This implies that 
$$
P(a_j')\leq \dfrac{C}{1+C}(1-|a_{jn}|^2).
$$
Hence, $P(b_j)=\dfrac{1}{1-|a_{jn}|^2}P(a_j')\leq \dfrac{C}{1+C}<1$ and thus $\psi_j(a_j)=(b_j,0)\to (b,0)\in D_P$ with $P(b)<1$. Therefore, Lemma \ref{interiorpoint-sq} yields $\liminf_{j\to \infty}\sigma_{D_P}(a_j)>0$, which is absurd.

\smallskip

\noindent
{\bf Case 2.} The sequence $\{a_j\}$ does not converge $\Lambda$-nontangentially to $p$. Then $P(a_j')\geq \; c_j \; \mathrm{dist}(a_j, \partial D_P)$ for some sequence $\{c_j\}\subset \mathbb R$ with $0<c_j\to +\infty$. This implies that $P(a_j')\geq c_j'(1-|a_{jn}|^2-P(a_j'))$ for some sequence $\{c_j'\}\subset \mathbb R$ with $0<c_j'\to +\infty$ and hence
$$
P(a_j')\geq \dfrac{c_j'}{1+c_j'}(1-|a_{jn}|^2), \forall j\geq 1.
$$
This tells us that $P(b_j)=\dfrac{1}{1-|a_{jn}|^2}P(a_j')\geq \dfrac{c_j'}{1+c_j'}, \forall j\geq 1$.
Therefore, we arrive at the situation  $b_j\to b$ with $P(b)=1$ and thus $\psi_j(a_j)$ converges to the strongly pseudoconvex boundary point $(b,0)$ of $\partial D_P$, which implies that $\sigma_{D_P}(a_j)=\sigma_{D_P}(\psi_j(a_j)) \to 1$ as $j\to \infty$ again by Theorem \ref{tendto1}. This leads to a contradiction. 

Hence, altogether, we complete the proof of Theorem \ref{Sqtheorem}.

\bigskip

\begin{acknowledgement}Part of this work was done while the first author was visiting the Vietnam Institute for Advanced Study in Mathematics (VIASM). He would like to thank the VIASM for financial support and hospitality. The third author was supported by the Vietnam National Foundation for Science and Technology Development (NAFOSTED) under grant number 101.99-2019.326. It is a pleasure to thank Hyeseon Kim for stimulating discussion.
\end{acknowledgement}

\bibliographystyle{plain}

\end{document}

%% file: Thu.pdf_tex
%% Creator: Inkscape 1.0 (4035a4fb49, 2020-05-01), www.inkscape.org
%% PDF/EPS/PS + LaTeX output extension by Johan Engelen, 2010
%% Accompanies image file 'Thu.pdf' (pdf, eps, ps)
%%
%% To include the image in your LaTeX document, write
%%   \input{<filename>.pdf_tex}
%%  instead of
%%   \includegraphics{<filename>.pdf}
%% To scale the image, write
%%   \def\svgwidth{<desired width>}
%%   \input{<filename>.pdf_tex}
%%  instead of
%%   \includegraphics[width=<desired width>]{<filename>.pdf}
%%
%% Images with a different path to the parent latex file can
%% be accessed with the `import' package (which may need to be
%% installed) using
%%   \usepackage{import}
%% in the preamble, and then including the image with
%%   \import{<path to file>}{<filename>.pdf_tex}
%% Alternatively, one can specify
%%   \graphicspath{{<path to file>/}}
%% 
%% For more information, please see info/svg-inkscape on CTAN:
%%   http://tug.ctan.org/tex-archive/info/svg-inkscape
%%
\begingroup%
  \makeatletter%
  \providecommand\color[2][]{%
    \errmessage{(Inkscape) Color is used for the text in Inkscape, but the package 'color.sty' is not loaded}%
    \renewcommand\color[2][]{}%
  }%
  \providecommand\transparent[1]{%
    \errmessage{(Inkscape) Transparency is used (non-zero) for the text in Inkscape, but the package 'transparent.sty' is not loaded}%
    \renewcommand\transparent[1]{}%
  }%
  \providecommand\rotatebox[2]{#2}%
  \newcommand*\fsize{\dimexpr\f@size pt\relax}%
  \newcommand*\lineheight[1]{\fontsize{\fsize}{#1\fsize}\selectfont}%
  \ifx\svgwidth\undefined%
    \setlength{\unitlength}{453.54330709bp}%
    \ifx\svgscale\undefined%
      \relax%
    \else%
      \setlength{\unitlength}{\unitlength * \real{\svgscale}}%
    \fi%
  \else%
    \setlength{\unitlength}{\svgwidth}%
  \fi%
  \global\let\svgwidth\undefined%
  \global\let\svgscale\undefined%
  \makeatother%
  \begin{picture}(1,1)%
    \lineheight{1}%
    \setlength\tabcolsep{0pt}%
    \put(0,0){\includegraphics[width=\unitlength,page=1]{Thu.pdf}}%
    \put(0.25022273,0.94214296){\color[rgb]{0,0,0}\makebox(0,0)[lt]{\lineheight{1.25}\smash{\begin{tabular}[t]{l}$\color{red}\psi\circ\Lambda_\lambda(\Omega)$\end{tabular}}}}%
    \put(0.13055098,0.89209833){\color[rgb]{0,0,0}\makebox(0,0)[lt]{\lineheight{1.25}\smash{\begin{tabular}[t]{l}$w'$\end{tabular}}}}%
    \put(0.52400113,0.8877467){\color[rgb]{0,0,0}\makebox(0,0)[lt]{\lineheight{1.25}\smash{\begin{tabular}[t]{l}$z'$\end{tabular}}}}%
    \put(0.96825314,0.42429066){\color[rgb]{0,0,0}\makebox(0,0)[lt]{\lineheight{1.25}\smash{\begin{tabular}[t]{l}$\text{Re } z_n$\end{tabular}}}}%
    \put(0.9182086,0.58312769){\color[rgb]{0,0,0}\makebox(0,0)[lt]{\lineheight{1.25}\smash{\begin{tabular}[t]{l}$\color{blue}\widetilde{\Sigma}$\end{tabular}}}}%
    \put(0.93126373,0.71150284){\color[rgb]{0,0,0}\makebox(0,0)[lt]{\lineheight{1.25}\smash{\begin{tabular}[t]{l}$\color{red} E^{r'}$\end{tabular}}}}%
    \put(0.91367722,0.80774622){\color[rgb]{0,0,0}\makebox(0,0)[lt]{\lineheight{1.25}\smash{\begin{tabular}[t]{l}$E^r$\end{tabular}}}}%
    \put(0.81027193,0.86160784){\color[rgb]{0,0,0}\makebox(0,0)[lt]{\lineheight{1.25}\smash{\begin{tabular}[t]{l}$E$\end{tabular}}}}%
    \put(0.62882061,0.43516989){\color[rgb]{0,0,0}\makebox(0,0)[lt]{\lineheight{1.25}\smash{\begin{tabular}[t]{l}$\color{red}D(r')$\end{tabular}}}}%
    \put(0.73630762,0.63639259){\color[rgb]{0,0,0}\makebox(0,0)[lt]{\lineheight{1.25}\smash{\begin{tabular}[t]{l}$\color{red}\Omega$\end{tabular}}}}%
    \put(0.73326139,0.404708){\color[rgb]{0,0,0}\makebox(0,0)[lt]{\lineheight{1.25}\smash{\begin{tabular}[t]{l}$\color{blue} D(r)$\end{tabular}}}}%
    \put(0.51863552,0.44937794){\color[rgb]{0,0,0}\makebox(0,0)[lt]{\lineheight{1.25}\smash{\begin{tabular}[t]{l}$0$\end{tabular}}}}%
    \put(0.65057907,0.50262127){\color[rgb]{0,0,0}\makebox(0,0)[lt]{\lineheight{1.25}\smash{\begin{tabular}[t]{l}$r'$\end{tabular}}}}%
    \put(0.69192023,0.50262127){\color[rgb]{0,0,0}\makebox(0,0)[lt]{\lineheight{1.25}\smash{\begin{tabular}[t]{l}$r$\end{tabular}}}}%
    \put(0,0){\includegraphics[width=\unitlength,page=2]{Thu.pdf}}%
    \put(0.38947711,0.78217449){\color[rgb]{0,0,0}\makebox(0,0)[lt]{\lineheight{1.25}\smash{\begin{tabular}[t]{l}$\psi=\psi^{-1}$\end{tabular}}}}%
    \put(0.17624379,0.09370061){\color[rgb]{0,0,0}\makebox(0,0)[lt]{\lineheight{1.25}\smash{\begin{tabular}[t]{l}$D_P$\end{tabular}}}}%
    \put(0.17406797,0.27737722){\color[rgb]{0,0,0}\makebox(0,0)[lt]{\lineheight{1.25}\smash{\begin{tabular}[t]{l}$D^r$\end{tabular}}}}%
    \put(0.17519943,0.3743552){\color[rgb]{0,0,0}\makebox(0,0)[lt]{\lineheight{1.25}\smash{\begin{tabular}[t]{l}$D^{r'}$\end{tabular}}}}%
    \put(0.11427565,0.54163918){\color[rgb]{0,0,0}\makebox(0,0)[lt]{\lineheight{1.25}\smash{\begin{tabular}[t]{l}$\Sigma$\end{tabular}}}}%
    \put(0.30679483,0.44680831){\color[rgb]{0,0,0}\makebox(0,0)[lt]{\lineheight{1.25}\smash{\begin{tabular}[t]{l}$1$\end{tabular}}}}%
    \put(0.39165297,0.44387331){\color[rgb]{0,0,0}\makebox(0,0)[lt]{\lineheight{1.25}\smash{\begin{tabular}[t]{l}$w_n$\end{tabular}}}}%
    \put(0.181727,0.4485653){\color[rgb]{0,0,0}\makebox(0,0)[lt]{\lineheight{1.25}\smash{\begin{tabular}[t]{l}$0$\end{tabular}}}}%
    \put(0.30679481,0.12402363){\color[rgb]{0,0,0}\makebox(0,0)[lt]{\lineheight{1.25}\smash{\begin{tabular}[t]{l}$\widetilde{\Sigma} = \psi(\Sigma)$\\$D^r = \psi(E^r)$\end{tabular}}}}%
    \put(0,0){\includegraphics[width=\unitlength,page=3]{Thu.pdf}}%
  \end{picture}%
\endgroup%